\documentclass[11pt, letter, parskip=full]{article}
\usepackage{amsmath,amssymb, amsfonts,amsthm, enumitem, babel, graphicx, url, }
\usepackage[hidelinks]{hyperref}
\usepackage[margin=1.0in]{geometry}
\usepackage[title]{appendix}
\newcommand{\Q}{\mathbb{Q}}
\newcommand{\R}{\mathcal{R}}
\newcommand{\No}{\mathcal{N}}
\newcommand{\Z}{\mathbb{Z}}

\newcommand{\logp}{\operatorname{log}_p}
\newcommand{\C}{\mathbb{C}}

\newcommand{\LegSym}[2]{\left(\frac{#1}{#2}\right)}

\newcommand{\Gal}{\operatorname{Gal}}
\newcommand{\w}{\omega}

\newcommand{\eps}{\varepsilon}
\newcommand{\AAC}{\text{Ankeny--Artin--Chowla}}
\newcommand{\dia}[1]{\langle #1 \rangle}
\newcommand{\Mod}[1]{\left(\mathrm{mod}\ #1\right)}

\theoremstyle{plain}
\newtheorem{thm}{Theorem}[section]
\newtheorem{lem}[thm]{Lemma}

\newtheorem{prop}[thm]{Proposition}
\newtheorem{cor}[thm]{Corollary}
\theoremstyle{definition}
\newtheorem*{rmk}{Remark}

\newtheorem*{defn}{Definition}
\usepackage{fancyhdr}

\title{\vspace*{-55pt}Congruence relations of $\AAC$ type for real quadratic fields}
\author{Nic Fellini\\Queen's University\\ \texttt{n.fellini@queensu.ca}}
\date{}

\begin{document}

\maketitle

\begin{abstract}
    In 1951, Ankeny, Artin, and Chowla published a brief note containing four congruence relations involving the class number of $\Q(\sqrt{d})$ for positive squarefree integers $d\equiv 1 \Mod{4}$. Many of the ideas present in their paper can be seen as the precursors to the now developed theory of cyclotomic fields. Curiously, little attention has been paid to the cases of $d\equiv~ 2,3\Mod{4}$ in the literature. In the present work, we show that the congruences of the type proven by Ankeny, Artin, and Chowla can be seen as a special case of a more general methodology using Kubota--Leopoldt $p$-adic $L$-functions.
    Aside from the classical congruence involving Bernoulli numbers, we derive congruences involving quadratic residues and non-residues in $\Z/p\Z$ by relating these values to a well known expression for $L_p(1, \chi)$.  We conclude with a discussion of known counterexamples to the so-called \textit{Composite $\AAC$ conjecture} and relate these to special dihedral extensions of $\Q$. 
\end{abstract}

\section{Introduction}
In 1951, Ankeny, Artin, and Chowla derived four congruence relations involving the fundamental unit and the class number of $\Q(\sqrt{d})$ for positive squarefree integers $d\equiv 1\Mod{4}$ \cite{Ankeny1951}. The proofs of three of the four results appear in the 1952 Annals of Mathematics paper \cite{Ankeny1952}. The most commonly cited of these congruences is the following: suppose $p\equiv 1 \Mod{4}$ and that $\eps = \frac{1}{2}( t + u\sqrt{p})$ is the fundamental unit of $\Q(\sqrt{p})$.  Then,
\begin{equation}\label{eq: Bernoulli--Class number}
 \frac{hu}{t} \equiv B_{\frac{p-1}{2}} \Mod{p}   
\end{equation}
where $h$ is the class number of $\Q(\sqrt{p})$ and $B_n$ is the $n$-th Bernoulli number defined by
\[
\frac{x}{e^x-1} = \sum_{n=0}^\infty B_n\frac{x^n}{n!}.
\]
We call the congruence in (\ref{eq: Bernoulli--Class number}) a \textit{congruence of $\AAC$ type}. It seems that several years earlier, Russian mathematician A.A. Kiselev proved  (\ref{eq: Bernoulli--Class number}) in \cite{Kiselev1948}.

The history and development of a congruence of $\AAC$ type for primes congruent to $3\Mod{4}$ are significantly less clear.  In 1961, using methods somewhat similar to Ankeny, Artin, and Chowla, Mordell proved that if $p\equiv 3 \Mod{4}$ and $t+u\sqrt{p}$ is the fundamental unit of $\Q(\sqrt{p})$, then
$p$ divides $u$ if and only if $p$ divides the $\left(\frac{p-3}{2}\right)$-Euler number \cite{mordell2}. Here, we define the $n$-th Euler number as the coefficient of $x^n/n!$ in the Taylor series 
\[
\frac{1}{\cosh x} = \sum_{n=0}^\infty E_n \frac{x^n}{n!}.
\]
We note that the $E_n$ defined above differ from the Euler numbers used by Mordell.  We will indicate in Section 3.1 the relation between the two definitions. 
However, Mordell seems to have ``missed the forest for the trees" and stated this weaker result, when in fact a congruence relation similar to that of (\ref{eq: Bernoulli--Class number}) was lurking in the background (see Corollary \ref{cor: Mordell} below). 

On the other hand, a congruence of $\AAC$ type for $p\equiv  3\Mod{4}$ has been attributed to Kiselev in \cite{Slavutskii2004}. However, the bibliographic information of where this congruence is in the literature is vague and, due to geopolitical reasons, much of Kiselev's work is inaccessible today. More recently, Y. Benmerieme and A. Movahhedi \cite{Benmerieme2024} derive such a relation but the connection to Euler numbers is not explored.

The original proof of Ankeny, Artin, and Chowla can be seen to contain the precursors of two important ideas in number theory, namely the $p$-adic logarithm and the use of group rings in the study of cyclotomic fields (see \cite{fellini2023} for this perspective). L. Washington gives a proof of congruence (\ref{eq: Bernoulli--Class number}) using the Kubota--Leopoldt $L$-function in Theorem 5.37 of \cite{Washington}. As it happens, Washington's proof extends virtually unchanged to deal with the case of $\Q(\sqrt{d})$ for $d>1$ squarefree. Moreover, using a formula derived by Washington in his construction of the $p$-adic $L$-function, we will show that Theorem 4 of \cite{Ankeny1951} can be proven in a straightforward manner. The connection to $p$-adic $L$-functions will allow us to easily deduce an analogous result for primes congruent to $3\Mod{4}$. 

We now explain the significance of the congruence (\ref{eq: Bernoulli--Class number}). Suppose that $p\equiv 1\Mod{4}$ and that
    \[
    \eps= \frac{1}{2}\left(t+u\sqrt{p} \right)
    \]
is the fundamental unit of $\Q(\sqrt{p})$, i.e., $(t,u)$ is the minimal\footnote{Here, minimal means that $t+\sqrt{p}u>1$ and is minimal among all other positive solutions when thought of as real numbers.} solution to the equation $x^2-py^2=\pm 4$.  It is plain that we must have $\gcd(t, p)=1$ as $p$ is odd. Moreover, it is known  that for these real quadratic fields $h<p$ \cite{AC, Le1994}. Combining these two facts, we deduce that if $u\not \equiv 0 \Mod{p}$ then the congruence gives a method for computing the class number without using Dedekind's class number formula\footnote{although, the class number formula is an essential part of the proof.}! 

In view of the congruence (1) Ankeny, Artin, and Chowla asked whether it was the case that $u$ is not divisible for primes $p\equiv 1\Mod{4}$. This has come to be known as the Ankeny--Artin--Chowla conjecture. About 10 years later, Mordell conjectured that $u\not \equiv 0 \Mod{p}$ for odd primes $p\equiv 3\Mod{4}$. We will refer to this pair of conjectures as the \textbf{$\AAC$--Mordell}  \textbf{conjecture} or the AACM conjecture for short. The AACM conjecture has been verified for all primes $2< p <  1.5\cdot 10^{12}$ \cite{Sidorov2024}.  Washington has heuristically argued that for $p\equiv 1\Mod{4}$ if the Bernoulli numbers $B_{\frac{p-1}{2}}$ behave ``randomly" $\Mod{p}$, then the expected number of counterexamples to the AACM conjecture (for primes $p\equiv 1\Mod{4}$) less than $x$ is asymptotic to $1/2\log\log x$ \cite{Washington}. However, as Washington remarks in the form of an exercise (see Exercise 5.9 \cite{Washington}), the hypothesis that Bernoulli numbers behave randomly $\Mod{p}$ is possibly asking for quite a lot. Despite this, the $\log \log (x)$ asymptotic seems to be directionally correct; in a sequence of two papers \cite{reinhart2024, reinhart_counterexample_2024}, A. Reinhart has found counterexamples to the AACM conjecture for $p=39028039587479 \equiv 3\Mod{4}$ and for $p=331914313984493\equiv 1\Mod{4}$. 

For composite square free $d>1$, we may be inclined to ask a similar question: given a composite squarefree integer $d$ and the fundamental unit $\eps = \frac{1}{2}( t+ u\sqrt{d})$  of $\Q(\sqrt{d})$, does $d\mid u$? We will refer to this as the \textbf{Composite AACM conjecture} or CAACM for short. The CAACM seems to have been first investigated by J. Yu and J.K. Yu in \cite{Yu1998}, although similar computations without the connection to the CAACM conjecture were done earlier (see \cite{Stephens1988}). The motivation for their investigation stems from their proof that the function field analogues of the AACM and CAACM conjectures are true! 

If we assume that $u$ behaves randomly $\Mod{d}$, then the number of counterexamples to the CAACM conjecture less than $x$ is given by
\[
\sum_{\substack{d\leq x\\ d \text{ squarefree}}} \frac{1}{d} =  \sum_{d\leq x } \frac{\mu^2(d)}{d} = \frac{6}{\pi^2}\log x + O(1). 
\]
Comparing with the computational results of \cite{reinhart2024, Stephens1988,Yu1998}\footnote{Yu and Yu omitted $d=430$ and $d=4099215$.}, we see that this heuristic closely agrees with the number of known counterexamples. For the largest search bound of $5.325\times 10^{13}$, there are 22 known counterexamples, whereas the heuristic suggest there would be approximately 19 counterexamples. It appears to be unknown if the CAACM conjecture is false for infinitely many squarefree positive $d$. Recent work of Akta\c{s} and Murty  suggests that a conjecture of Erd\"{o}s implies if counterexamples exist to the AACM or CAACM conjecture, then they are rare (in the sense that they have natural density zero) \cite{aktas}. Using a simple sieve argument, it was shown in \cite{fellini2023} that there are infinitely many squarefree $d$ for which the CAACM conjecture is true.

\subsection{Statement of results}
Let $\chi$ be a Dirichlet character of conductor $f$ and set
    \[
        F_{\chi}(t) = \sum_{a=1}^f \frac{\chi(a)te^{at}}{e^{ft}-1}.
    \]
Expanding $F_{\chi}(t)$ as formal power series we have:
\[
    F_{\chi}(t) = \sum_{n=0}^\infty B_{n, \chi} \frac{t^n}{n!}.
\]
We call the coefficients $B_{n, \chi}$ of $F_{\chi}(t)$ \textit{generalized Bernoulli numbers}.

Let $p$ be an odd prime and write $d=pm$ for a squarefree positive integer $m$ coprime to $p$.  Let $D$ be the fundamental discriminant of $\Q(\sqrt{d})$, $\chi_D(\cdot) = \LegSym{D}{\cdot}$ denote the Kronecker symbol of conductor $D$,  and $\psi_m (n):=\LegSym{n}{p}\chi_D(n)$ be a quadratic character with conductor $D/p$.  
\begin{thm}\label{thm: 1}
    Let $d=pm$ be a positive squarefree integer, $p$ an odd prime, and $k=\Q(\sqrt{d})$. We write $D=\delta^2d$, $h$, and $\eps = \frac{\delta}{2}(t+u\sqrt{d} )$ for the discriminant, class number and fundamental unit of $k$ respectively. 
    \begin{enumerate}[label=(\Alph*)]
        \item  If $p>3$, then 
        \[
            \frac{hu}{t}\equiv \delta B_{\frac{p-1}{2},\psi_m} \Mod{p}.
        \]
        
        \item If $p=3$ and  $m\equiv 1\Mod{3}$, then  
        \[
        \frac{hu}{t} \equiv  \delta  B_{1, \psi_m} \Mod{3}.
        \] 
    \end{enumerate}
\end{thm}
The two simplest cases of Theorem \ref{thm: 1} (A) are the congruences mentioned in the introduction. Indeed, if $d=p\equiv 1\Mod{4}$ is prime, then $D=p$ and $\psi_1= \chi_1$ the principal character. As such, we deduce the congruence in (\ref{eq: Bernoulli--Class number}):
\begin{cor}[$\AAC$  and Kiselev] \label{cor: AAC}
    If $p\equiv 1\Mod{4}$, then 
    \[
    \frac{hu}{t} \equiv B_{\frac{p-1}{2}} \Mod{p}
    \]
    where $h$ and $\eps=\frac{1}{2} ( t + u\sqrt{p})$ are the class number and fundamental unit of $\Q(\sqrt{p})$ respectively. 
\end{cor}
For $d=p\equiv 3\Mod{4}$, $D=4p$, and $\psi_m  = \chi_{-1}$, the non-trivial character of conductor $4$. It is classically known that $B_{n, \chi_{-1}}$ is related to the Euler number $E_{n-1}$ (see for example \cite{euler_numbers}). In this case, we obtain the Euler number\footnote{Some authors instead define Euler numbers as the coefficients in the series expansion for $1/\cos x$. These are related by a factor of $(-1)^{n/2}$ for $n$ even to the Euler numbers we have defined.} appearing in Mordell's work and the congruence attributed to Kiselev in \cite{Slavutskii2004}\footnote{The stated congruence seems to be off by a factor of $1/2$.}:
\begin{cor}[Kiselev, Mordell] \label{cor: Mordell}
    If $p\equiv 3\Mod{4}$, then 
    \[
    \frac{hu}{t} \equiv\frac{1}{2} E_{\frac{p-3}{2}} \Mod{p}
    \]
    where  $h$ and $\eps=( t + u\sqrt{p})$ are the class number and fundamental unit of $\Q(\sqrt{p})$ respectively and $E_{n}$ is the $n$-th coefficient in the expansion of $1/\cosh(t)$. 
\end{cor}

A special case of Theorem 1.1 (B) appears in \cite{Ankeny1952} when $3m\equiv 1\Mod{4}$ and the statement of Theorem 1.1 (B) appears an exercise in \cite{Washington}. We remark that the condition $m\equiv 1 \Mod{3}$ is necessary for the methods we have employed in this work. Indeed, if $m\equiv 2\Mod{3}$, then $3$ splits in the imaginary quadratic field $\Q(\sqrt{-m})$ and in effect, the $3$-adic $L$-function, $L_3(s, \chi_{D})$, has a trivial zero at $s=0$.  This in turn, implies that $L_3(n, \chi_D)$ is not a $3$-adic unit for \textit{any} integer $n$ (see also: Chapitre IV of \cite{gras_1977}). As such, the congruence one obtains simply reads: $0\equiv 0 \Mod{3}$. A better understanding of the coefficients of the power series for $L_3(s, \chi)$ would provide a method to examine this case more thoroughly.

Using the well known fact that the class number $H$ of an imaginary quadratic field $\Q({\sqrt{-d}})$ for $d>4$ is equal to $-B_{1, \chi_D}$ we deduce from Theorem \ref{thm: 1}(B): 
 \begin{cor}\label{cor: p=3}
     Suppose $d=3m$ where $m\equiv 1\Mod{3}$ is a positive integer. If $H$ is the class number of $\Q(\sqrt{-m})$, then 
     \[
     \frac{hu}{t} \equiv -\delta H \Mod{3},
     \] 
     where  $h$ and $\eps=\frac{\delta}{2}( t + u\sqrt{p})$ are the class number and fundamental unit of $\Q(\sqrt{3m})$. 
 \end{cor}
 
By relating the congruences to $p$-adic $L$-functions we obtain a generalization of Theorem 4 in \cite{Ankeny1951}. Given the quadratic character $\chi_D$ of $\Q(\sqrt{d})$, we define two sets:
$\R=\{1\leq r \leq D:\chi_D(r)=1\}$ and $\No=\{1\leq n \leq D :\chi_D(n)=-1\}$ and form the two quantities:
\[
R = \prod_{r\in \R} r\,\,\,\,\, \text{ and } \,\,\,\,\, N=\prod_{n\in \No} n. 
\]

\begin{thm}\label{thm: 2}
     Let $d=pm$ be a positive squarefree integer, $p$ an odd prime, and $k=\Q(\sqrt{d})$. If $p=3$, assume further that $m\equiv 1\Mod{3}$. 
     \begin{enumerate}[label=(\Alph*)]
    \item If $R\equiv N\equiv 1\Mod{D}$, then
    \[
    \frac{2\delta hu}{t} \equiv \frac{N-R}{d}\Mod{p}.
    \]
    \item If $R\equiv N\equiv -1\Mod{D}$, then 
    \[
    \frac{2\delta hu}{t}\equiv \frac{R-N}{d}\Mod{p}.
    \]
    \item If $R\equiv - N\equiv -1\Mod{D}$, then 
    \[
    \frac{2hu}{t}\equiv \frac{R+N}{d}\Mod{p}. 
    \]
\end{enumerate}
\end{thm}
\noindent In the case that $d=p\equiv 1\Mod{4}$ is prime, we obtain Theorem 4 of \cite{Ankeny1951}:

\begin{cor}
    If $d=p\equiv 1\Mod{4}$ is prime, 
    \[
    \frac{2hu}{t} \equiv \frac{R+N}{p} \Mod{p}. 
    \]
\end{cor}

If $d=p\equiv 3\Mod{4}$, then there are four solutions to $x^2=1\Mod{4p}$, namely, $1, 2p-1, 2p+1$, and $4p-1$. Both $2p-1$ and $2p+1$ belong to $\No$. From this we deduce:
\begin{cor}
    If $d=p\equiv 3\Mod{4}$ is prime, then 
    \[
    \frac{4hu}{t} \equiv \frac{R-N}{p} \Mod{p}. 
    \]
\end{cor}
\begin{rmk}
    Lemma \ref{lem: num order 2} and Lemma \ref{lem: sign of A mod D} show that the three cases above encompass all of the possible cases for $R$ and $N$.  Case (A) occurs when $D$ contains at least three prime factors or if $d=pq$ for distinct primes $p$ and $q$ such that $p\equiv q\equiv 1\Mod{4}$. Case (B) occurs if $d=p$ for some prime $p\equiv 3\Mod{4}$ of if $D=pq$ for distinct primes $p$ and $q$ such that $p\equiv q\equiv 3\Mod{4}$. Case (C) only occurs if $d=p$ for some prime $p\equiv 1\Mod{4}$.
\end{rmk}
\begin{rmk}
    One should compare the proof we give of Theorem \ref{thm: 2} (C) with the proof in \cite{Carlitz1953}. The essential point in both proofs is the $p$-adic logarithm, although its appearance in \cite{Carlitz1953} is not as obvious. See Section 5 of \cite{fellini2023} for this connection. 
\end{rmk}
In addition to the new congruences above, we obtain a relation between the CAACM conjecture and the zeros of certain $p$-adic $L$-functions.
\begin{thm}\label{thm: p-adic unit}
    Let $d>2$ be a squarefree integer. The CAACM conjecture is true if and only if for any odd prime divisor of $p$ of $d$, $L_p(s, \chi_D)$ does not vanish for any $s$ in its domain of definition.  Equivalently, the CAACM conjecture is true if and only if $L_p(1, \chi_D)$ is a $p$-adic unit for any odd prime divisor $p$ of $d$. 
\end{thm}
In the case that $d=p$ is an odd prime, the above theorem reduces to the non-vanishing of a single $p$-adic $L$-function. 

Our final result is an extension of recent works of \cite{Benmerieme2024} and \cite{Cohen2020} relating the AACM conjecture to the existence of certain dihedral extensions of $\Q$. Here we call an extension $L/\Q$ a dihedral extension if $\Gal(L/\Q) \cong D_n$, the dihedral group of order $2n$. Our result is as follows:

\begin{thm}\label{thm: dihedral}
    The CAACM (resp. AACM) conjecture is true for $K=\Q(\sqrt{d})$ if and only if for any odd prime divisor $p$ of $d$, there does not exist a tower of extensions $L\supseteq K \supseteq \Q$ such that:
        \begin{enumerate}[label=(\roman*)]
        \item $\Gal(L/\Q) \cong D_p$
        \item $L/K$ is unramified outside of the primes lying above $p$ in $K$.
    \end{enumerate}
\end{thm}
\begin{rmk}
    We note that if such a tower of extension exists, then $K$ is the unique quadratic subfield of $L$. In particular, the only rational primes that ramify in $L$ are those dividing the discriminant $D$ of $K$. Moreover, for all primes $\ell\neq p$ and $\ell \mid D$, the ramification index of $\ell$ in $L$ is $2$. 
\end{rmk}

In the process of this work, there were several natural questions that arose that seem deserving of further study:
 
\begin{quote}
    \textbf{Question 1.} For any real quadratic field $\Q(\sqrt{d})$ and prime divisor $p$ of $d$, does there exist a constant $\alpha$, independent of $d$ and $p$, such that $p^{\alpha}\nmid u$? The case of $\alpha =1$ is the CAACM and AACM conjecture. \\

    \textbf{Question 2.} Is the CAACM conjecture (unconditionally) false for infinitely many $d$? The question of whether there are infinitely many counterexamples to the CAACM conjecture does not seem to appear anywhere in the literature.\\

    \textbf{Question 3.} Can a precise asymptotic be obtained for the number of counterexamples to either the AACM or CAACM?\\

    \textbf{Question 4.} Can the $p$-adic methods used here be used to prove similar congruences for other totally real fields? Some results in this direction for pure cubic, quartic, and sextic fields can be found here in \cite{ito1984, kamei_1987}. 
\end{quote}

\section{$p$-adic preliminaries}
For a non-zero integer $n$, we define $v_p(n)$ to be the largest power of $p$ dividing $n$. If $n=0$, we set $v_p(0)=\infty$. We denote by $|\cdot|_p$ the standard $p$-adic metric on $\Q$. Thus if $x=a/b$, $a,b\in \Z$, $\gcd(a,b)=1$ and $b\neq 0$, then 
\[
|x|_p = p^{v_p(b)-v_p(a)}.
\]
We denote by $\Z_p$ and $\Q_p$ the completions of $\Z$ and $\Q$ (respectively) with respect to $|\cdot|_p$.  Taking the completion of the algebraic closure of $\Q_p$ we obtain $\C_p$, an algebraically closed complete metric space analogous to the complex numbers.  Here the metric is extended from $\Q_p$ to $\C_p$ and normalized so that $|p|_p =p^{-1}$. The non-zero elements of $\C_p$ can be decomposed as 
\[
\C_p^\times  = p^{\Q}\times W \times U_1
\]
where $p^\Q$ is the set $\{p^r : r\in \Q\}$, $W$ is the set of roots of unity with order coprime to $p$ and $U_1 = \{x\in \C_p : |x-1|_p<1\}$ (see for example  Proposition 5.4 of \cite{Washington} or section III.4 of \cite{Koblitz}. 

\subsection{The $p$-adic Logarithm}
One of the crucial tools in the work of Ankeny, Artin, and Chowla is the $p$-adic logarithm. 
\begin{defn}
    The \textit{$p$-adic logarithm }is defined by the formal power series
    \[
    -\log_p(1-x) = \sum_{n=1}^\infty \frac{x^n}{n}.
    \]
    Additionally, as a formal power series $\log_p$ also satisfies the functional equation
    \[
    \log_p(xy)=\log_p(x) + \log_p(y). 
    \]
\end{defn}
\noindent This series defining the $p$-adic logarithm can be easily see to converge $p$-adically for all $x\in U_1$. Moreover, $p$-adic logarithm can be (\textit{uniquely}) extended to all of $\C_p^\times$ by declaring that $\log_p(p)=\log_p(w)=0$ where $w$ is any root of unity. Further details can be found in \cite[Proposition 5.4]{Washington}. We record the following:
\begin{prop}
    Given $\alpha \in \C_p^\times$, we write $\alpha = p^rwx$ where $w\in W$ and $x\in U_1$. We set $\log_p(p)=\log_p(w)=0$. Then $\log_p$ defined above can be uniquely extended to all of $\C_p^\times$ by the rule 
    \begin{align*}
        \log_p: \C_p^\times &\to \C_p\\
        p^rwx &\mapsto \log_p x.
    \end{align*}
    Moreover, this extension satisfies the functional equation $\log_p(\alpha \beta) = \log_p(\alpha) + \log_p(\beta)$. 
\end{prop}

We can evaluate the $p$-adic logarithm of elements in $\Q(\sqrt{d})$ using the following identity in $\C_p$. 
\begin{prop}\label{prop: p-adic log of fund. unit}
    Fix an odd prime $p$. Suppose $d=pm$ is a positive squarefree integer. Moreover, suppose $\eps$ is the fundamental unit of $\Q(\sqrt{d})$. In $\C_p$ we have,  
    \[
    \frac{\log_p(\eps)}{\sqrt{d}} = \left(\frac{u}{t}\right)\sum_{n=0}^\infty \frac{d^{n}}{2n+1}\left(\frac{u}{t} \right)^{2n}.
    \]
\end{prop}
\begin{proof}
    We will prove this in the case that $\eps = \frac{1}{2} (t+u\sqrt{d})$ satisfies $t^2-du^2=\pm 4$. The case when $t^2-du^2=\pm 1$ follows an identical argument. We write
    \[
    \frac{\log_p(\eps)}{\sqrt{d}} = \frac{1}{\sqrt{d}}\left(\log_p\left( \frac{t}{2} \right) + \log_p\left(1 +\frac{u\sqrt{d}}{t} \right) \right). 
    \]
    Noting that as $p$ is odd, that $(p,t)=1$, we see that $|(u/t)\sqrt{d}|_p<1$ and we can expand the logarithm term as:
    \[
   \frac{1}{\sqrt{d}}\log_p\left(1 +\frac{u\sqrt{d}}{t} \right) = \sum_{n=1}^\infty \frac{(-1)^{n-1}}{n} \left(\frac{u}{t}\right)^n d^{(n-1)/2}.  
    \]
    The above series is absolutely convergent so  we split the sum into two as follows:
    \[
    \sum_{n=1}^\infty \frac{(-1)^{n-1}}{n} \left(\frac{u}{t}\right)^n d^{(n-1)/2} = \sum_{n=1}^\infty \frac{1}{2n-1}\left(\frac{u}{t} \right)^{2n-1}d^{n-1} - \frac{1}{2\sqrt{d}}\sum_{n=1}^\infty \frac{1}{n} \left(\frac{u}{t} \right)^{2n}d^{n}. 
    \]
    We see that the second sum is precisely:
    \[
    \sum_{n=1}^\infty \frac{1}{n} \left(\frac{du^2}{t^2} \right)^{n} = -\logp\left(1- \frac{du^2}{t^2}\right) = -\logp\left(\frac{t^2-du^2}{t^2}\right) = 2\log_p\left(\frac{t}{2} \right).
    \]
    Here, we have used the fact that $\logp(xy) = \logp(x)+\logp(y)$ and that $\logp(\pm 1) = 0$. Therefore, 
    \[
    \frac{1}{\sqrt{d}}\log_p\left(1 +\frac{u\sqrt{d}}{t} \right) = -\frac{\logp(t/2)}{\sqrt{d}}+\sum_{n=1}^\infty \frac{1}{2n-1}\left(\frac{u}{t} \right)^{2n-1}d^{n-1}. 
    \]
    The proposition follows upon changing the index of summation. 
\end{proof}

\subsection{The Teichm\"{u}ller Character}

We begin with a remark regarding Dirichlet characters in the $p$-adic context. Since a Dirichlet character $\chi$ is valued in the roots of unity and $0$, we must assign $p$-adic values to $\chi$ to discuss it in the context of this work. To achieve this, we choose an arbitrary but fixed embedding of $\overline{\Q}$ into $\C_p$ and thus consider $\chi(a)$ as an element in $\C_p$.   

We note that for odd primes $p$, the non-zero elements of $\Z_p$ are isomorphic (as abelian groups) to 
\[
\Z_p^\times  \cong \Z_p\cap (W \times U_1 ) \cong (\Z/(p-1)\Z) \times (1+\Z_p).
\]
In particular, given any $a\in \Z_p^\times$ we can write $a= \omega(a)\langle a\rangle$ where $\omega(a)$ is the $(p-1)$-st root of unity such that $\omega(a)\equiv a \Mod{p}$ and $\dia{a} \equiv 1\Mod{p}$. We call $\omega(a)$ the \textit{Teichm{\"u}ller} representative of the residue class $a \Mod{p}$. We define the Teichm{\"u}ller character in the following way:
\begin{defn}
    The map
    \begin{align*}
        (\Z/p\Z)^\times \to \Z_p\\
        a \mapsto \omega(a)
    \end{align*}
    is called the Teichm{\"u}ller character. If $x\in \Z_p$, we extend this function to $\Z_p$ by setting $\w(x) := \w(a)$ where $x\equiv a \Mod{p}$.  
\end{defn} 

Following the remark at the beginning of this section, it is plain that the Teichm{\"u}ller character is a Dirichlet character of conductor $p$. Moreover, for odd primes $p$, $\w^{\frac{p-1}{2}}$ is the Legendre symbol $\Mod{p}$. Indeed, if $a\in (\Z/p\Z)^\times$, then $\w(a)^{p-1}-1=0$. As $p-1$ is even, we have that:
\[
\w^{\frac{p-1}{2}}(a) = \pm 1.
\]
Furthermore, since $\w(a)\equiv a\Mod{p}$, raising both sides to the $\left(\frac{p-1}{2}\right)$-power we obtain:
\[
\w^{\frac{p-1}{2}}(a) \equiv a^{\frac{p-1}{2}} \Mod{p}.
\]
By Euler's criteria, this is $1\Mod{p}$ if $a$ is a non-zero quadratic residue $\Mod{p}$ and $-1 \Mod{p}$ if $a$ is a quadratic non-residue $\Mod{p}$. Therefore,
\[
\w^{\frac{p-1}{2}}(a) = \LegSym{a}{p}
\]
if $a\in (\Z/p\Z)^\times$. We can extend $\w$ to all of $\Z_p$ by defining $\w(0)=0$ and for any $x\in \Z_p$ with $x\equiv a\Mod{p}$ we have that $\w(x) = \w(a)$.

\subsection{The Kubota--Leopoldt $L$-function}

\begin{thm}[Washington \cite{washington1976}]\label{thm: p-adic L function}
Let $\chi$ be a Dirichlet character of conductor $f$ and let $F$ be any multiple of $p$ and $f$. Then there exists a $p$-adic meromorphic (analytic if $\chi \neq \chi_0$) function $L_p(s, \chi)$ on $\{s\in \C_p: |s|_p < p\cdot p^{-1/(p-1)}\}$ such that 
\[
L_p(1-n, \chi) = - ( 1 - \chi\omega^{-n}(p)p^{n-1}) \frac{B_{n, \chi\omega^{-n}}}{n}
\]
for $n\geq 1$. If $\chi$ is the trivial character, then $L_p(s, \chi)$ has a simple pole at $s=1$ with residue $1-p^{-1}$. Moreover, we have the formula
\[
L_p(s, \chi)= \frac{1}{F} \frac{1}{s-1} \sum_{\substack{a=1\\ (a, p)=1}}^F \chi(a)\langle a \rangle^{1-s}\sum_{j=0}^\infty \binom{1-s}{j} \left(\frac{F}{a} \right)^j B_j
\]
where $\binom{x}{n} = \frac{1}{n!}x( x-1) \cdots (x-n+1)$ for $n>0$ and $\binom{x}{0}=1$. 
\end{thm}

The Kubota--Leopoldt $L$-function satisfies a rather remarkable congruence relation at the positive integers. Namely:
\begin{prop}\label{cor: equal mod p}
    Suppose $\chi \neq 1$ and $p^2\nmid f_\chi$. Let $m,n\in \Z$. Then 
    \[
    L_p(m,\chi) \equiv L_p(n, \chi) \Mod{p}.
    \]
\end{prop}
The previous congruence is fundamental in proving our results. Indeed, it allows us to relate the value $L_p(1, \chi_D)$ to $L_p(1- \frac{p-1}{2}, \chi_D)$. To prove Theorem \ref{thm: 2}, we require the following congruence which is an easy consequence of Theorem \ref{thm: p-adic L function} and Proposition \ref{cor: equal mod p}.
\begin{cor}\label{cor: 3.11}
    If $\chi$ is a non-trivial even Dirichlet character of conductor $f$ and $F$ is any multiple of $p$ and $f$ then 
    \[
    L_p(1, \chi) \equiv \frac{-1}{F}\sum_{\substack{a=1\\p\nmid a}}^F \chi(a)\logp(a)\Mod{p}.
    \]
    Moreover, suppose $d>1$ is squarefree and $p>3$ is a prime such that $d=pm$. Let $\chi_D$ be the quadratic character of $\Q(\sqrt{d})$ and $\psi_m = \LegSym{\cdot}{p}\chi_D$. Then, 
    \[
    L_p(1 -\frac{p-1}{2}, \chi_D) \equiv 2B_{\frac{p-1}{2}, \psi_m} \equiv -\frac{1}{D}\sum_{\substack{a=1\\p\nmid a}}^D \chi(a)\logp(a) \Mod{p}.
    \]
\end{cor}

We conclude this section with the $p$-adic class number formula for real quadratic fields (see Theorem 5.24 of \cite{Washington} for a more general result). 
\begin{thm}\label{thm: p-adic class number formula}
    Suppose $K$ is a real quadratic field with discriminant $D$, fundamental unit $\eps$ and class number $h$. Then, 
    \[
    \frac{2h\log_p\eps}{\sqrt{D}} = \left(1 -\frac{\chi_D(p)}{p} \right)^{-1}L_p(1, \chi_D).
    \]
\end{thm}
\begin{rmk}
    Both $\log_p \eps$ and $\sqrt{D}$ are only determined up to sign. As such, we get the above by fixing signs so as to obtain equality. 
\end{rmk}

\section{Proof of Theorem \ref{thm: 1}}

\begin{proof}
Let $p$ be an odd prime and $d=pm$ be a squarefree positive integer. Let $K=\Q(\sqrt{d})$, $D=\operatorname{disc}(K)$, $\chi_D$ the quadratic character associated to $K$, and $\psi_m(a) = \LegSym{a}{p}\chi_D(a)$ be a Dirichlet character of conductor $D/p$. 

Fixing embeddings so as not to have any ambiguity in signs, the $p$-adic class number formula yields: 
\[
\frac{2h\logp \eps }{\sqrt{D}}=\left(1- \frac{\chi_D(p)}{p}\right)^{-1}L_p(1, \chi).  
\]
As $\chi$ is an even Dirichlet character of conductor $D$ and $p\mid D$, we have $\chi_D(p)=0$. Therefore, the above expression reduces to 
\[
\frac{2h\logp \eps}{\delta \sqrt{pm}} = L_p(1, \chi)
\]
where $\delta =1$ if $d\equiv 1 \Mod{4}$ and $\delta =2$ if $d\equiv 2,3\Mod{4}$. 

It is at this point that the proof of the two parts of Theorem 1.1 diverge. Assume that $p>3$.
By Proposition \ref{prop: p-adic log of fund. unit} we have that:
\[
\frac{1}{\sqrt{d}}\log_p \eps \equiv \frac{u}{t} \Mod{p}. 
\]

On the other hand,  Corollary \ref{cor: equal mod p} implies that 
\[
L_p(1, \chi_D) \equiv L_p\left(1 - \frac{p-1}{2}, \chi_D\right) \Mod{p}.
\]
We note that $\w^{\frac{p-1}{2}}(a) = \LegSym{a}{p}$ and hence $\chi_Dw^{\frac{p-1}{2}} = \psi_m$. 
Theorem \ref{thm: p-adic L function} implies that 
\[
L_p\left(1-\frac{p-1}{2}, \chi_D \right)  = -\left(1-\psi_m(p)p^{\frac{p-3}{2}}\right) \frac{B_{\frac{p-1}{2}, \psi_m}}{\frac{p-1}{2}}.
\]
By assumption, $p>3$, so $\frac{p-1}{2}>1$ and thus reducing $\Mod{p}$ we obtain:
\[
L_p(1, \chi_D) \equiv 2B_{\frac{p-1}{2}, \psi_m} \Mod{p}.
\]
Comparing the two congruences we have obtained for $L_p(1, \chi_D)$, we deduce:
\[
\frac{hu}{t} \equiv \delta B_{\frac{p-1}{2}, \psi_m} \Mod{p}. 
\]

Now we assume that $p=3$. By Proposition \ref{prop: p-adic log of fund. unit} we have:
\[
\frac{1}{\sqrt{3m}}\log_3 \eps \equiv \frac{u}{t} + m\left(\frac{u}{t}\right)^3 \Mod{3}.
\]
We note that $(t,3)=1$, so $(u/t)^3 \equiv (u/t) \Mod{3}$. Therefore, 
\[
\frac{1}{\sqrt{3m}}\log_3 \eps \equiv \left(1+m\right)\frac{u}{t}\Mod{3}. 
\]
By assumption $m\equiv 1\Mod{3}$ so 
\[
\frac{2h\log_3 \eps}{\delta\sqrt{3m}} \equiv \frac{hu}{\delta t} \Mod{3}. 
\] 
On the other hand,  Corollary \ref{cor: equal mod p} implies that 
\[
L_3(1, \chi_D) \equiv L_3(0, \chi_D) \Mod{3}.
\]
We note that $\w(a) = \LegSym{a}{3}$ and hence $\chi_D\w = \psi_m$.
Theorem \ref{thm: p-adic L function} implies that 
\[
L_3\left(0, \chi_D \right)  = -(1-\psi_m(3))B_{1, \psi_m}.
\]
Since $m\equiv 1 \Mod{3}$, $3$ does not split in $\Q(\sqrt{-m})$. As such, $\psi_m(3)=-1$. From this, we deduce that 
\[
\frac{hu}{t} \equiv \delta B_{1, \psi_m} \Mod{3}. 
\]
\end{proof}

\section{Proof of Theorem \ref{thm: 2}}
Let $\R_2 = \{a\in \R : r^2\equiv 1\Mod{D}\} \subset \R$ and $\No_2 = \{n\in \No: n^2 =1 \Mod{D}\}\subset \No$. 
Noting that the multiplicative inverses $\Mod{D}$ of elements in $\R$ (resp. $\No$) are also in $\R$ (resp. $\No$), we deduce that 
\[
R \equiv \prod_{r\in \R_2} r \Mod{D}
\]
and 
\[
N \equiv \prod_{n\in \No_2}n \Mod{D}.
\]
As $\chi_D$ is an even character, we can pair up $c$ and $-c$ in the above products to deduce that:
\[
R\equiv (-1)^{|\R_2|/2}\Mod{D}
\]
and 
\[
N\equiv (-1)^{|\No_2|/2}\Mod{D}
\]
where we use the convention that $|\emptyset|=0$. 
To determine the values of $R$ and $N$ $\Mod{D}$, it suffices to determine the possible sizes of $\R_2$ and $\No_2$.

\begin{lem} \label{lem: num order 2}
Suppose that $\rho$ is the number of solutions to $x^2\equiv 1\Mod{D}$.
If $\No_2\neq \emptyset$, then $|\R_2| = |\No_2| = \rho/2$. Otherwise, $|\R_2|=\rho$. 
\end{lem}
\begin{proof}
    Suppose $\No_2\neq \emptyset$ and let $n\in \No_2$. Then the map
    \begin{align*}
            f : \R_2 &\to \No_2\\
            r &\mapsto rn \Mod{D}.
    \end{align*}
   is a bijection of finite sets and hence $|\R_2|= |\No_2|$.
\end{proof}
\begin{lem}\label{lem: sign of A mod D}
   If $d=p\equiv 1\Mod{4}$ is prime, $R\equiv-N \equiv -1\Mod{p}$. Otherwise, $R\equiv N \Mod{D}$. 
\end{lem}
\begin{proof}
    If $\rho=2$, then $D=p\equiv 1 \Mod{4}$ is prime. In this case $x\equiv \pm 1\Mod{p}$ are both in $\R_2$ and hence $R\equiv -1 \Mod{p}$. Since $\No_2$ is empty, $N\equiv 1\Mod{p}$. If $\rho=4$, then either $|\R_2|=|\No_2|=2$ or $|\R_2|=4$. In the first case $R\equiv N\equiv -1\Mod{D}$. In the latter case, $R\equiv N\equiv 1\Mod{D}$. We note that if $\rho>4$, then $R\equiv N\equiv 1\Mod{D}$.
\end{proof}

The author would like to thank Lawrence Washington for suggesting a simplified proof of Lemma~\ref{lem: sign of A mod D} in the case $d=p\equiv 3\Mod{4}$ which ultimately lead to the above proof.  

\begin{proof}[Proof of Theorem \ref{thm: 2}]
We will prove the theorem in case (A). Suppose $R\equiv N\equiv 1\Mod{D}$ and write 
\[
R = 1 - D\Omega 
\]
and 
\[
N = 1 - D\Omega^*
\]
for some integers $\Omega$ and $\Omega^*$. As $p\mid D$, we can expand the $p$-adic logarithms of $R$ and $N$ using the power series definition to obtain the congruences:
\[
\log_p(R) \equiv -D\Omega \Mod{p^2}
\]
and 
\[
\log_p(N)\equiv -D \Omega^* \Mod{p^2}. 
\]
Therefore, 
\[
\log_p(R)-\log_p(N) \equiv -D(\Omega -\Omega^*)\Mod{p^2}. 
\]
In particular, 
\[
-\frac{1}{D}\left(\log_p(R) -\log_p(N)\right) \equiv \Omega-\Omega^* \Mod{p}.
\]
Noting that $\frac{N-R}{D}= (\Omega-\Omega^*)$ we obtain
\[
-\frac{1}{D}\left(\log_p(R) -\log_p(N)\right) \equiv \frac{N-R}{D} \Mod{p}.
\]
Moreover, by Corollary \ref{cor: 3.11} the left hand side is precisely $L_p(1, \chi_D)\Mod{p}$, which is congruent to $2hu/\delta t\Mod{p}$ as seen in the proof of Theorem \ref{thm: 1}. Recalling that $D=\delta^2d$, we deduce: 
\[
\frac{2\delta hu}{t}\equiv \frac{N-R}{d} \Mod{p}. 
\]

To deduce (B) and (C) we proceed as above with the following changes:
If $R\equiv N\equiv -1\Mod{D}$, then we write
\[
R = -1 + D\Omega
\]
and 
\[
N =-1+D\Omega^*.
\]
If $R\equiv -1\Mod{p}$ and $B\equiv 1\Mod{p}$, then we write
\[
R = -1 +p\Omega
\]
and 
\[
N =1-p\Omega^*.
\]

\end{proof}
\section{Counterexamples to the Ankeny--Artin--Chowla conjecture}
As mentioned in the introduction, there are a total of $21$ known squarefree composite integers $d$ that are counterexamples to the CAACM conjecture and two counterexamples to the AACM conjecture. In this section we relate these conjectures to the zeros of $p$-adic $L$-functions.

By Iwasawa's construction of the $p$-adic $L$-function, there is a function $g_\chi \in \Z_p[[T]$ which we write as 
\[
g_\chi (T) = b_0 + b_1T +b_2T^2 + \cdots. 
\]
We call this the Iwasawa series of $\chi$. In particular, Iwasawa showed that 
\[
L_p(s, \chi) = g_\chi( (1+p)^s -1)
\]

\noindent By the $p$-adic Weierstrass preparation theorem (\cite{Washington} Theorem 7.3) we can write $g_\chi$ as 
\[
g_\chi(T) = p^\mu P(T)U(T)
\]
where $\mu\geq 0$ is an integer, $P(t)$ is of the form 
\[
P(T) = T^\lambda + c_{n-1}T^{\lambda -1} + \cdots + c_0
\]
where $p\mid c_i$ for each $0\leq i \leq \lambda-1$, and $U(T)$ is a unit in $\Z_p[[T]]$. We denote by $\lambda_p(\chi)$ the degree of $P(T)$ and call this the Iwasawa $\lambda$-invariant of $\chi$. As a consequence of a deep theorem of Ferrero and Washington \cite{FW1979}, we have that $\mu=0$ . 

It is known that in the domain of definition of $L_p(s, \chi)$, that $U((1+p)^s-1)$ has no zeros. As such, the zeros of $L_p(s, \chi)$ arise from the zeroes of $P(T)$ (see for example, \cite{Ernvall92} for the details). It should be noted that $P(T)$ can have zeros that are outside of the domain of $L_p(s, \chi)$. For an interpretation of these zeros, see \cite{Childress87}.

An elementary exercise in ring theory states that if $R$ is a ring (with an identity element), then $f\in R[[x]]$ is a unit if and only if $f(0)\in R^\times$ (see, for example, Exercise 7.1 \cite{Washington}). In particular, as $U(T)$ is a unit, we find that $U(0)$ is a unit in $\Z_p$. From this we see that 
\[
g_\chi(0) \in \begin{cases}
    \Z_p^\times & \text{if $\lambda_p(\chi) = 0$}\\
    p\Z_p &\text{if $\lambda_p(\chi) >0$}
\end{cases}
\]
In particular, we see that $\lambda_p(\chi)>0$ if and only if $\deg(P) >0$.

Returning to $p$-adic $L$-functions we see that 
\[
L_p(0, \chi) = g_\chi(0).
\]
Since $L_p(1, \chi_D) \equiv L_p(0, \chi_D) \Mod{p}$ by Corollary \ref{cor: equal mod p}, we deduce Theorem \ref{thm: p-adic unit} from the introduction.

\section{Dihedral extensions of $\Q$}
We now elaborate on a recent theme explored by \cite{Benmerieme2024} and \cite{Cohen2020} relating counterexamples to the AACM conjecture to the existence of certain dihedral extensions of $\Q$. We begin by introducing some notation. Suppose $p$ is a rational prime, $K$ is a number field, and $\Sigma$ is the set of primes lying above $p$ in $K$. We set 
\begin{align*}
   M &:= \text{ compositium of all finite $p$-extensions of $K$ that are unramified outside of $\Sigma$}, \\
   \Gamma &:= \Gal(M/K), \\
   M^{ab} &:= \text{ the maximal abelian extension of $K$ in $M$},\\
   \Gamma^{ab}&:= \Gal(M^{ab}/K).
\end{align*}
It is known that $\Gamma^{ab}$ is finitely generated as a $\Z_p$-module (see for example Chapter 13 section 5 of \cite{Washington}). Denoting by $r_2(K)$ the number of complex embeddings of $K$, we say $K$ is $p$-rational if 
\begin{enumerate}[label= (\arabic*)]
    \item $\operatorname{Rank}_{\Z_p}(\Gamma^{ab}) = r_2(K)+1$.
    \item $\Gamma^{ab}$ is torsion free as a $\Z_p$-module. 
\end{enumerate}

\noindent Condition (1) is Leopoldt's conjecture for the prime $p$ and field $K$. In particular, (1) is always satisfied for abelian extensions of $\Q$ (see Theorem 5.25 of \cite{Washington}). As explained in \cite{Greenberg2016}, if $K$ is a totally real Galois extension of $\Q$, $K$ is $p$-rational if and only if $M^{ab} = K\Q_{\infty}$, the cyclotomic $\Z_p$-extension of $K$. Moreover, if $K$ is totally real, $\Gamma = \Gamma^{ab} \cong \Z_p$ and hence $M=K\Q_{\infty}$. We record this as a lemma for use later.
\begin{lem}\label{lem: Zp - p-rational}
    A totally real field $K$ is $p$-rational if and only if the maximal extension of $K$ that is unramified outside the primes above $p$, is the cyclotomic $\Z_p$-extension of $K$. 
\end{lem}

The connection between $p$-rationality of real quadratic fields and $p$-adic $L$-functions comes from the following:
\begin{thm}[Corollary 2.2 of \cite{Benmerieme2021}]
Suppose $K=\Q(\sqrt{d})$ for a squarefree positive integer $d$ and that $D$ is the discriminant of $K$. Then $K$ is $p$-rational if and only if $L_p(1, \chi_D)$ is a $p$-adic unit.    
\end{thm}
By Theorem \ref{thm: p-adic unit} we deduce:
\begin{cor}\label{cor: p-rational - AAC}
    The CAACM conjecture is true for $K=\Q(\sqrt{d})$ if and only if for some prime divisor $p\mid d$, $K$ is $p$-rational. 
\end{cor}

We will call a Galois extension $L/\Q$ a $D_n$-extension if $\Gal(L/\Q) \cong D_n$, the dihedral group of order $2n$. We related the $p$-rationality of of real quadratic fields to the existence of certain dihedral extensions of $\Q$.

\begin{thm}
    Let $K=\Q(\sqrt{d})$ be a real quadratic field and $p>2$ be a prime divisor of $d$. Then $K$ is $p$-rational if and only if there does not exist a tower of extensions $L\supset K \supset \Q$ such that 
    \begin{enumerate}[label=(\roman*)]
        \item $\Gal(L/\Q) \cong D_p$
        \item $L/K$ is unramified outside of the primes lying above $p$ in $K$.
    \end{enumerate}
\end{thm}
\begin{proof}
    For sake of contradiction, suppose that $K$ is $p$-rational and that $L/K/\Q$ is a tower of extensions satisfying (i) and (ii). Since $K$ is totally real and assumed to be $p$-rational, Lemma \ref{lem: Zp - p-rational} implies that $M=K\Q_\infty$. This extension is abelian over $\Q$ (as it is the compositium of two abelian extensions). Since $L/K$ has degree $p$ and is unramified outside of $\Sigma$, it must be contained in $M$. But $M$ is abelian over $\Q$ and hence $L$ must be abelian over $\Q$. Since $D_p$ is non-abelian for $p>2$ no such extension can exist.  

    We now show that if $K$ is not $p$-rational then an extension $L$ satisfying (i) and (ii) exist. Suppose $K$ is not $p$-rational, that is, $M \neq M^{ab}$. Let $F$ be the Galois closure (over $\Q$) of any minimal non-abelian $p$-extension of $K$ contained in $M$. Such a field must exist as $M$ is non-abelian and the Galois closure of any such non-Galois extension is also in $M$.  By the Galois Correspondence, $\Gal(F/K)$ is a normal subgroup of $\Gal(F/\Q)$. We let the non-trivial element of $\Gal(K/\Q)$ act on $\Gal(F/K)$ by conjugation. Since $F/K$ is non-abelian, this action is non-trivial. Combining Corollary 3.2 and Lemma 4.1 of \cite{MazurRubin2008}, we deduce that $\Gal(F/K)$ contains a normal subgroup $H$ such that $|\Gal(F/K) : H| =p$ and conjugation induces the automorphism $x\to x^{-1}$ on the quotient $\Gal(F/K)/H$. Let $L=F^H$ be the fixed field of $H$ in $F$. Then $L$ is a degree $p$ Galois extension of $K$ with Galois group $G= \Gal(F/K)/H$. In particular, $\Gal(K/\Q)$ is normal in $G$ and the action by conjugation induces the automorphism $x\to x^{-1}$ of $G$. We deduce that $\Gal(L/\Q) = \Gal(K/\Q) \rtimes \Gal(F/K)/H \cong D_p$. As $L\subseteq F \subseteq M$, $L$ must be unramified outside of the primes above $p$ in $K$. 
\end{proof}
\noindent Combining the previous theorem and Corollary \ref{cor: p-rational - AAC} we deduce Theorem \ref{thm: dihedral}. \\

\vspace*{0.25cm}
\noindent
\textbf{Acknowledgments.} The author would like to thank Ram Murty, Lawrence Washington, Hershy Kisilevsky, and Antonio Lei for their helpful comments on an earlier versions of this work which helped improve the presentation. The author would also like to thank  Abhishek Bharadwaj their helpful comments during the preparation of this manuscript.  
\renewcommand\refname{References}
\bibliography{References}

\end{document}